\documentclass[12pt, reqno]{amsart}
\usepackage{amsmath, amsthm, amscd, amsfonts, amssymb, graphicx }
\usepackage{mathtools}
\usepackage{blkarray, bigstrut} %

\newtheorem{theorem}{Theorem}[section]

\newtheorem{corollary}[theorem]{Corollary}
\theoremstyle{definition}

\newtheorem{example}[theorem]{Example}

\theoremstyle{remark}
\newtheorem{remark}[theorem]{Remark}
\numberwithin{equation}{section}

\begin{document}
	
	\setcounter{page}{1}
	\title[Property $(UW_E)$ and Consistent invertibility]{ Property $(UW_E)$ and Consistent invertibility }
	\author[  simi thomas, t prasad   and shery fernandez]{ simi thomas, t prasad   and shery fernandez }
	
	\address{Simi Thomas \endgraf Department of Mathematics\endgraf Cochin University of science and technology  \endgraf
		Kerala, \endgraf
		India -682022}
	\email{simimariumthomas@cusat.ac.in }
	\address{T Prasad  \endgraf Department of Mathematics\endgraf University of Calicut \endgraf
		Kerala, \endgraf
		India- 673635.}
	\email{prasadvalapil@gmail.com }
	\address{Shery Fernandez \endgraf Department of Mathematics\endgraf Cochin University of science and technology  \endgraf
		Kerala, \endgraf
		India -682022}
	\email{sheryfernandez@cusat.ac.in }
	\dedicatory{ }

	
	
	%\date{Received: xxxxxx; Revised: yyyyyy; Accepted: zzzzzz.}
	%\newline \indent $^{*}$Corresponding author}
\begin{abstract}
	In this paper,  we study the stability of Property $(UW_E)$ under commuting finite rank perturbation in terms of spectrum induced by consistent in invertibility for operators and a variant of upper semi-Weyl spectrum. Also we discuss the stability of Property $(UW_{E})$ under direct sum $T\oplus S$ and tensor product $T\otimes S$ for Hilbert space operators $T$ and $S$.
\end{abstract}\maketitle
\section{Introduction and Preliminaries}
Throughout this paper, $B(\mathcal{H})$ and $K(\mathcal{H})$ denotes the algebra of all bounded linear operators and the ideal of all compact operators on infinite dimensional complex separable Hilbert space $\mathcal{H}$ respectively. The adjoint, the kernel, the range, the nullity, and the defect for an operator $T \in B(\mathcal{H})$ are denoted by $T^*$, $\mathcal{N}(T)$, $\mathcal{R}(T)$, $\alpha(T)$, and $\beta(T)$ respectively. An operator $T$ is said to be an upper semi-Fredholm (resp., lower semi-Fredholm) operator if $\alpha (T)< \infty $ (resp., $\beta(T)< \infty$) and $R(T)$ is closed. Let 
\begin{eqnarray*}
	\Phi_+(\mathcal{H})&=&\lbrace T \in B(\mathcal{H}): T \text{ is upper semi-Fredholm} \rbrace ~\text{and }\\
	 \Phi_-(\mathcal{H})&=&\lbrace T \in B(\mathcal{H}): T \text{ is lower semi-Fredholm} \rbrace 
\end{eqnarray*}
be the collection of all upper and lower semi-Fredholm operators.
   An operator $T$ is called semi-Fredholm, $  T \in SF(\mathcal{H}) $, if $T \in \Phi_+(\mathcal{H}) \cup \Phi_-(\mathcal{H})$. Let $\Phi(\mathcal{H})= \Phi_-(\mathcal{H})\cap\Phi_+(\mathcal{H})$ denote the collection of all Fredholm operators. Index of an operator $T\in SF(\mathcal{H})$ is given by 
   \begin{eqnarray*}
   i(T)&=& \alpha(T)-\beta(T).
   \end{eqnarray*}
    An operator $T \in B(\mathcal{H})$ is Weyl if it is Fredholm of index zero. Let
\begin{eqnarray*}
SF_+(\mathcal{H}) & =& \lbrace T \in SF: i(T)\leq0 \rbrace;\\
SF_-(\mathcal{H})  &=& \lbrace T \in SF: i(T)\geq 0 \rbrace 
\end{eqnarray*}
be the collection of all upper semi-Weyl operators and the collection of all lower semi-Weyl operators respectively.
Ascent of an operator $T$, $p:=p(T)$, and descent of an operator $T$, $q:=q(T)$, is given by
\begin{eqnarray*}
p&=& \text{min} \lbrace n\in \mathbb{N} \cup \lbrace\infty \rbrace : \mathcal{N}(T^n)=\mathcal{N}(T^{n+1})\rbrace ~\text{and}\\
q&=& \text{min} \lbrace n\in \mathbb{N} \cup \lbrace\infty \rbrace : R(T^n)=R(T^{n+1})\rbrace.
\end{eqnarray*}
  If ascent and descent are both finite, then $p(T)=q(T)$\cite[Theorem 1.20]{aiena}. Fredholm operators with finite ascent and descent is termed as Browder operator. We denote
\begin{center}
$BR(\mathcal{H})= \lbrace T \in B(\mathcal{H}): T ~\text{is Browder} \rbrace.$
\end{center}
 An operator $T\in B(\mathcal{H})$ is said to be Drazin invertible if $p(T)=q(T)<\infty$. Let \begin{center}
 	$DI(\mathcal{H})= \lbrace T \in B(\mathcal{H}): T ~\text{is Drazin invertible} \rbrace.$ \end{center} 
It is well known that Drazin invertible operators with finite nullity is Browder. The essential spectrum $\sigma_{e}(T)$, the Wolf spectrum $\sigma_{lre}(T)$, the Weyl spectrum $\sigma_w(T)$, the upper semi-Weyl spectrum $\sigma_{uw}(T)$, the lower semi-Weyl spectrum $\sigma_{lw}(T)$,   the Browder spectrum $\sigma_{b}(T)$ and the Drazin invertible spectrum $\sigma_{dz}(T)$ of $T \in B(\mathcal{H})$ are given by
\begin{eqnarray*}
	\sigma_e(T) & = & \lbrace \lambda \in \mathbb{C} : T- \lambda I \notin \Phi(\mathcal{H}) \rbrace, \\
	\sigma_{lre}(T) &=& \lbrace \lambda \in \mathbb{C} : T- \lambda I \notin SF(\mathcal{H}) \rbrace,\\
	\sigma_w(T) &=& \lbrace \lambda \in \mathbb{C}  : T-\lambda I \text{\,\,\,is not Weyl} \rbrace,\\
	\sigma_{uw}(T) &=& \lbrace \lambda \in \mathbb{C}  : T-\lambda I \notin SF_+(\mathcal{H}) \rbrace,\\
	\sigma_{lw}(T) &=& \lbrace \lambda \in \mathbb{C}  : T-\lambda I \notin SF_-(\mathcal{H}) \rbrace, \\
	\sigma_{b}(T) &=& \lbrace \lambda \in \mathbb{C}  : T-\lambda I \notin BR(\mathcal{H}) \rbrace,\\
	\sigma_{dz}(T) &=& \lbrace \lambda \in \mathbb{C}  : T-\lambda I \notin DI(\mathcal{H}) \rbrace.
\end{eqnarray*}
The semi-Fredholm domain of $T$ is denoted as $\rho_{SF}(T)= \mathbb{C} \setminus \sigma_{lre}(T)$ and let 
\begin{eqnarray*}
	\rho_{SF}^+(T)&:=&\lbrace \lambda \in \rho_{SF}(T) : i(T-\lambda I) >0\rbrace,\\
	\rho_{SF}^-(T)&:=&\lbrace \lambda \in \rho_{SF}(T) : i(T-\lambda I) <0\rbrace, \\
	\rho_{SF}^0(T)&:=&\lbrace \lambda \in \rho_{SF}(T) : i(T-\lambda I) =0\rbrace.
\end{eqnarray*}
It is evident that $\sigma_{w}(T)= \mathbb{C}\setminus \rho_{SF}^0(T)$ and $\rho_{SF}(T)=\rho_{SF}^0(T) \cup \rho_{SF}^+(T) \cup \rho_{SF}^-(T)$. Given $\emptyset \neq \sigma \subseteq \sigma(T)$, a clopen subset. Then there exists a $\Omega$, which is an analytic Cauchy domain such that $\sigma\subseteq \Omega$ and $[\sigma(T)\setminus \sigma] \cap \overline{\Omega} = \emptyset$. The operator $E(\sigma; T)$ is defined by
\begin{center} 
	$E(\sigma; T) = \frac{1}{2\pi i}\int_{\Gamma}(T-\lambda I)^{-1}d\lambda$,
 \end{center}
where $\Gamma = \partial \Omega$ is positively oriented with respect to $\Omega$. The operator
 $E(\sigma; T)$ commutes with every operator that commutes with $T$ often called Riesz idempotent operator. Let $H(\sigma; T)= R(E(\sigma; T))$. We simply write $H(\lambda; T)$ instead of $H(\lbrace \lambda \rbrace; T)$. An isolated point $\lambda$ of $\sigma(T)$ is called a normal eigenvalue
of $T$ if $\text{dim} H(\lambda; T) <\infty$. The set of all normal eigenvalues of $T$ will be denoted as $\sigma_0(T)$. Let $iso~\sigma(T)$ denote the isolated points of $\sigma(T)$.
We denote 
\begin{eqnarray*}
E(T)&:=&\lbrace \lambda \in iso~\sigma(T) :\alpha (T-\lambda I) >0\rbrace~\text{and}\\
E_0(T)&:=&\lbrace \lambda \in iso~\sigma(T) :0<\alpha (T-\lambda I) <\infty\rbrace
\end{eqnarray*}
If for every open disc $D$ centered at $\lambda$, the only analytic function $f : D \rightarrow X$ that satisfies the equation $(\lambda I - T)f(\lambda) = 0$ for all $\lambda \in D$ is the function $f \equiv 0$, then $T \in B(\mathcal{H})$ is said to have the single valued extension property at $\lambda \in \mathbb{C}$ (abbreviated T has SVEP at $\lambda$).
\\
%In spectral theory, Weyl-type theorems and its variants are a significant area of study.
 Herman Weyl \cite{weyl} examined that the intersection of spectrum of all compact perturbation of self adjoint operators are exactly the points in the spectrum that are not isolated eigenvalues of finite multiplicity. This was soon extended to normal operators. 
According to Coburn \cite{cob}, an operator $T$ is said to satisfy Weyl's theorem, if its Weyl spectrum $\sigma_w(T)$ consists of  those points in the spectrum $\sigma(T)$ that are not isolated eigenvalues of finite multiplicity. That is,
\begin{center}
	$\sigma(T)\setminus\sigma_{w}(T)=E_0(T)$.
\end{center} 
 Coburn \cite{cob} studied Weyl's theorem for nonnormal operators and showed that Weyl's theorem holds for classes of hyponormal operators and Toeplitz operators. Weyl-type theorems and its variations has been studied for a variety of operators and studied by many authors on both Hilbert spaces and Banach spaces \cite{aiena,apon,azn,berk,rashid,sana,zhou,zhu}. Rakočević \cite{racko} introduced and studied  $a$- Weyls theorem. An operator $T$ satisfies $a$- Weyls theorem if
\begin{center}
	$\sigma_a(T) \setminus  E_{0}^a(T)=\sigma_{uw}(T)$.
\end{center}
An operator  $T$ is said to satisfy  property $(w)$  if 
 \begin{center}
$\sigma_a(T)\setminus  E_0(T)=\sigma_{uw}(T)$.
\end{center}
  According to \cite{sana} an operator $T$ is said to satisfy property $(V_E)$ if 
  \begin{center}
  	$\sigma(T) \setminus  E(T)=\sigma_{uw}(T)$.
  \end{center} 
Property $(UW_E)$, a variant of Weyl-type theorem has been introduced by Berkani and Kachad \cite{berk}, studied by many authors \cite{azn,li,pra,qiu,sana,simi,sun}. Let $T \in B(\mathcal{H})$. Then $T$ satisfies property $(UW_{E})$, if 
\begin{center}
	$\sigma_a(T)\setminus  E(T)=\sigma_{uw}(T)$.
\end{center}  Property $(w)$ does not implies $(UW_E)$ (see example 3.2). Qiu and Cao \cite{qiu} studied some examples to show that property ($UW_E$) and $a$- Weyls theorem are independent. In \cite{pra,simi}, the authors give sufficient conditions for functions of operators satisfying property $(UW_E)$ and $T+K $ satisfying property $(UW_E)$ for all compact operators $K$. 

In this paper, we explored certain conditions for the stability of property $(UW_E)$ satisfied by the operators under commuting finite rank perturbation using the spectrum induced by consistent in invertibility and also the stability of property $(UW_E)$ under tensor product and direct sum under certain conditions. 

\section{Consistent Invertibility for property $(UW_E)$}
Consistent in invertibility for operators has been studied by Hladnik and Omladic \cite{hladnik} and later it was studied by many authors see \cite{djor,gong,ren,xin}. An operator $T\in B(\mathcal{H})$ is said to be consistently invertible if for each $S \in B(\mathcal{H})$, $TS$ and $ST$ are either both or neither invertible. Gong and Han \cite{gong} given that all normal, compact and invertible operators are consistently invertible and they characherize it for some nonnormal operators. Let
\begin{center}
	$CI(\mathcal{H})= \lbrace T \in B(\mathcal{H}): T ~\text{is consistently invertible} \rbrace.$\\
\end{center}
The consistent invertibilty spectrum $\sigma_{CI}(T)$ and is given by
\begin{center} 
	$\sigma_{CI}(T) = \lbrace \lambda \in \mathbb{C}  : T-\lambda I \notin CI(\mathcal{H}) \rbrace.$
\end{center}
It is evident that $\sigma_{CI}(T)$ is contained in $\sigma_{b}(T)$.

The following example shows that property $(UW_E)$ may not be stable under commuting finite rank perturbations.
\begin{example}
	Let  \begin{equation*}
		T = 
		\begin{pmatrix}
			T_1 &  0\\
			0 & I 
		\end{pmatrix}, K = 
		\begin{pmatrix}
			0 &  0\\
			0 & K_1 
		\end{pmatrix}
	\end{equation*} on $l^2 \oplus l^2$, where $T_1(x_1,x_2,x_3,..)=(0,x_1,\frac{x_2}{2},\frac{x_3}{3},...)$ and $K_1(x_1,x_2,x_3,..)=(-x_1,0,0,...)$. Then\begin{center}
	$\sigma_a(T+K)=E(T+K)=\lbrace 0,1\rbrace$. 
\end{center} 
Since, $0\in\sigma_{uw}(T+K)$, it follows that $T+K$ does not satisfies property $(UW_E)$. 
\end{example} Let $\sigma_{vv}(T= \mathbb{C} \setminus \rho_{vv}(T)$, where
%\begin{center}\

 $\rho_{vv}(T):=\lbrace \lambda \in \mathbb{C} : \text{there exists an}\rm~ \epsilon> 0 \rm~ \text{such that} \rm~ T-\mu I \in W_+(\mathcal{H}) \rm~ \vspace{.25 cm} and  \rm~ \mathcal{N}(T-\mu I)\subseteq \bigcap_{n=1}^{\infty} \mathcal{R}(T-\mu I)^n $  if $\rm~ 0<|\lambda-\mu|< \epsilon\rbrace$.
%\end{center}\\
Then 
\begin{center}
	$\sigma_{vv}(T) \subseteq \sigma_{uw}(T) \subseteq \sigma_b(T) \subseteq \sigma(T)$.
\end{center} 
Xin and Jiang studied the stability of property $(w)$ under commuting finite rank perturbation. See \cite{xin} for more deatils. In a similar manner, 
 we study the stability of property $(UW_E)$ under commuting finite rank perturbation by using the variant of the upper semi-Weyl spectrum $\sigma_{vv}(T)$ and the consistent invertibilty spectrum $\sigma_{CI}(T)$.
%The next theorem shows that property $(UW_E)$ transferred to finite rank perturbation under certain conditions.
\begin{theorem}
	Suppose that $T \in B(\mathcal{H})$ satisfies property $(UW_E)$. If $K$ is a finite rank operator such that $[K,T]=0$, then the following are equivalent.
	\begin{enumerate}
	\item[(i)] $T + K$ is isoloid and satisfies property $(UW_E)$. 
	\item[(ii)] $T$ is isoloid and $T+K$ satisfies property $(UW_E)$. 
	\item[(iii)] $\sigma_b(T ) \cap \sigma_a(T + K) = \sigma_{vv}(T ) \cup [\overline{\sigma_{CI} (T )} \cap \sigma_a(T )]$.
\end{enumerate}
\end{theorem}
\begin{proof}
	(i)$\Longrightarrow$(ii): Suppose $T+K$ is isoloid and satisfies property $(UW_E)$. It is enough to prove $T$ is isoloid. Let $\lambda \in \text{iso}~\sigma(T)$ but $\alpha(T-\lambda I)=0$. Since $\lambda \in \text{iso}~\sigma(T)$ and $K$ is of finite rank, we have $\lambda \in \text{iso}~\sigma(T+K) \cup \rho(T+K)$. Since $T+K$ is isoloid, $\lambda \in E(T+K)=\Pi(T+K)$. Thus $T+K-\lambda I \in DI(\mathcal{H})$. Since $T+K-\lambda I \in  W_+(\mathcal{H})$, $T+K-\lambda I$ is Browder. Which implies $\lambda \notin \sigma_b(T)$. Then $T-\lambda I$ is invertible, a contradiction. Hence  $T$ is isoloid.\\
	(ii)$\Longrightarrow$(iii): $T$ is isoloid and $T+K$ satisfies property $(UW_E)$.Then, from \cite[Theorem 2.1]{xin}, we have 
	\begin{center}
		$\sigma_b(T ) \cap \sigma_a(T + K) \supseteq \sigma_{v}(T ) \cup [\overline{\sigma_{CI} (T )} \cap \sigma_a(T )]$.
	\end{center}
	Since $\sigma_{v}(T)\subseteq \sigma_{vv}(T)$, we obtain
	\begin{center}
		$\sigma_b(T ) \cap \sigma_a(T + K) \supseteq \sigma_{vv}(T ) \cup [\overline{\sigma_{CI} (T )} \cap \sigma_a(T )]$.
	\end{center}
	 Let $\lambda_0\notin\sigma_{vv}(T ) \cup [\overline{\sigma_{CI} (T )} \cap \sigma_a(T )]$. Then there exist an $\epsilon>0$ such that 
	$T-\lambda I \in W_+(\mathcal{H}) \rm~  and \rm~ \mathcal{N}(T-\mu I)\subseteq \bigcap_{n=1}^{\infty} \mathcal{R}(T-\mu I)^n\rm~  if \rm~ 0<|\lambda-\lambda_0|< \epsilon$. Since property $(UW_E)$ hold for $T$ and $T-\lambda I$ is bounded below if $0<|\lambda-\lambda_0|<\epsilon$,  we have to consider two cases, $\lambda_0 \notin \overline{\sigma_{CI} (T )}$ and $\lambda_0\notin\sigma_{a}(T)$.\\ If $\lambda_0 \notin \overline{\sigma_{CI} (T )}$, then $T-\lambda I$ is invertible for small $|\lambda-\lambda_0|>0$. Then $\lambda_0 \in \text{iso}\rm~\sigma(T)\cup\rho(T)$. This implies $\lambda_0\in \text{iso}\rm~\sigma(T+K)\cup\rho(T+K)$. Since $T+K$ is isoloid, $\lambda_0 \in E(T+K)$. Thus $T+k-\lambda_0 I$ is Browder as same in the first part. Then $T-\lambda_0I$ is Browder. That is $\lambda_0\notin\sigma_b(T)\cap \sigma_a(T+K)$.\\
	 If $\lambda_0\notin\sigma_{a}(T)$, then either $\lambda_0\notin\sigma_{a}(T+K)$ or $\lambda_0\in\sigma_{a}(T+K)\setminus\sigma_{uw}(T+K)$. Suppose $\lambda_0\in\sigma_{a}(T+K)\setminus\sigma_{uw}(T+K)$. Then, $T+K-\lambda_0I$ is a Browder operator. Since $T\in(UW_E)$, $T-\lambda_0I$ is Browder. Thus, $\lambda_0\notin\sigma_b(T)\cap\sigma_{a}(T+K)$.\\
	(iii)$\Longrightarrow$(i): Assume that 
	\begin{center}
		
		$\sigma_b(T)\cap\sigma_{a}(T+K)=\sigma_{vv}(T)\cup[\overline{\sigma_{CI}}\cap\sigma_{a}(T)]$.
	\end{center}
	Let $\lambda_0\in \sigma_a(T+K)\setminus\sigma_{uw}(T+K)$. Then, $T+K-\lambda_0 I $ is upper semi-Weyl and consequently, $T-\lambda_0 I $ is upper semi-Weyl. If $\lambda_0 \in\sigma_a(T)$, then $T-\lambda_0 I$ is Browder . Thus $T+K-\lambda_0 I$ is Browder. If $\lambda_0\notin\sigma_a(T)$, then $$\lambda_0 \notin \sigma_{vv}(T)\cup [\overline{\sigma_{CI}(T)} \cap \sigma_a(T)].$$ Which implies that $\lambda_0\notin\sigma_b(T)\cap\sigma_a(T+K)$. Since $\lambda_0\in\sigma_{a}(T+K)$, $\lambda_0\in\sigma_b(T)$. Then $T+K-\lambda_0I$ is Browder. Which implies $\lambda_0\in E(T+K)$.
	Conversly, $\lambda_0\in E(T+K)$ implies $\lambda_0\in \text{iso}\sigma(T)\cup \rho(T)$. We have $\lambda_0 \in\text{iso}\sigma(T)$ implies $\lambda_0 \notin \sigma_{vv}(T)\cup [\overline{\sigma_{CI}(T)} \cap \sigma_a(T)]$. Thus $\lambda_0 \notin\sigma_b(T)$ . This implies $T+K-\lambda_0I$ is Browder. Hence $\lambda_0\in \sigma_a(T+K)\setminus\sigma_{uw}(T+K)$.
	\end{proof}
\begin{corollary}
	Let	$T \in B(\mathcal{H})$. Then $T$ is isoloid and satisfies property $(UW_E)$ if and only if $\sigma(T ) = \sigma_{vv}(T ) \cup [\overline{\sigma_{CI} (T )} \cap \sigma_a(T )]\cup [\sigma_b(T) \cap \rho_a(T)]\cup\sigma_0(T)$.
\end{corollary}
The quasi-nilpotent part of $T \in B(\mathcal{H})$ is defined as the set
\begin{center}
	$H_0(T)=\lbrace h:lim_{n\rightarrow\infty}||T^nh||^{\frac{1}{n}}=0\rbrace$
\end{center}
Aponte et.al \cite{aponte} characterized property $(V_{\pi})$ in terms of quasi-nilpotent part of an operator. The following theorem is a characterization of property $(UW_E)$ using the quasi-nilpotent part of $T$.

\begin{theorem}
	Let $T \in B(\mathcal{H})$. Then the following statements are equivalent:
	\begin{enumerate}
		
	\item[(i)]  $T$ satisfies property $(UW_E)$.
	\item[(ii)] For each $\lambda \in \sigma_a(T) \setminus \sigma_{uw}(T)$, there exists $ v:=v(\lambda) \in \mathbb{N}$ such that
	$H_0(T-\lambda I)= \mathcal{N}(T-\lambda I)^v$
	and $\Pi(T)=E(T)\subseteq \rho_{SF_-}(T)\cup\rho_{SF_0}(T)$.
	\item[(iii)] $H_0(T-\lambda I)$ is closed for all $\lambda \in \sigma_a(T) \setminus \sigma_{uw}(T)$ and $\Pi(T)=E(T)\subseteq \rho_{SF}^-(T)\cup\rho_{SF}^0(T)$.
\end{enumerate}
\end{theorem}
\begin{proof}
	\begin{enumerate}
	(i)$\Longrightarrow$(ii): Suppose $T \in (UW_E)$, Then by \cite[Lemma 3.2] {azn}, we have $E(T)=\Pi(T)$. If $\lambda \in \sigma_a(T)\setminus\sigma_{uw}(T)$, then $\lambda \in E(T)=\Pi(T) $. Thus $\lambda$ is a pole of the resolvent of $T$. Then it follows that there exists $ v:=v(\lambda) \in \mathbb{N}$ such that
	$H_0(T-\lambda I)= \mathcal{N}(T-\lambda I)^v$ by \cite[Theorem 2.47]{aiena}. Since $T\in(UW_E)$, it folloes that $E(T)\subseteq \rho_{SF_-}(T)\cup\rho_{SF_0}(T)$.\\
	(ii)$\Longrightarrow$(iii): Evident from the hypothesis.\\
	(iii)$\Longrightarrow$(i): Let $\lambda \in \sigma_a(T) \setminus \sigma_{uw}(T)$. Then $T-\lambda I$ is upper semi-Weyl. Since $H_0(T-\lambda I)$ is closed, $T$ has SVEP at $\lambda$. Then from  \cite[Remark 1.5]{apon}, we have $p(T-\lambda I)< \infty$.  Consequently, $q(T-\lambda I)<\infty$. Thus $\lambda \in \sigma(T) \setminus \sigma_{d}(T)=\Pi(T)=E(T)$. Thus $ \sigma_a(T) \setminus \sigma_{uw}(T)\subseteq E(T)$. Now $E(T)\subseteq [\rho_{SF_-}(T)\cup\rho_{SF_0}(T)] \cap \sigma_{p}(T)=\sigma_a(T) \setminus \sigma_{uw}(T)$. Hence $T$ satisfies property $(UW_E)$.
	\end{enumerate}
\end{proof}

%(1).	$T$ satisfies property $(UW_E)$ if and only if $T$ satisfies Weyl's theorem,  $\sigma_{w}(T)\setminus\sigma_{uw}(T)=\sigma(T)\setminus\sigma_{a}(T)$  and $E(T)=E_0(T)$.\\
 If $T\in (UW_E)$ and $\sigma(T) \neq \sigma_a(T)$, then there exist a $\lambda \in \rho_{SF}^-(T)$ such that $H_0(T-\lambda I)$ is closed and consequently $T$ has SVEP at $\lambda$. In this case property $(UW_E)$ is same as property $(W_E)$ and $(V_E)$.
 If $T$ satisfies property $(UW_E)$ if and if $T$ satisfies property $(V_{\pi})$ and $E(T)=\pi(T)$. Moreover,
 If $T$ is hyponormal, then $\sigma(T)\setminus\sigma_a(T)=\sigma_{CI}(T)$ \cite{cao}. So if $\sigma_{CI}(T)=\emptyset$, then property $(V_E)$ is equivalent to property $(UW_E)$. Sanabria et.al in \cite{sana} showed that property $(V_E)$ implies property $(UW_E)$ but the converse need not true. For example, consider the right shift operator  $R \in B(l^2)$ given by $R(x_1,x_2,x_3,...)=(0,x_1,x_2,x_3,...)$, then $\sigma(R)=\lbrace \lambda \in \mathbb{C}: |\lambda | \leq 1\rbrace$ and  $\sigma_a(R)= \sigma_{uw}(R), E(R)=\emptyset$. If an operator satisfies property $(UW_E)$, then it need not necessary that its adjoint satisfies property $(UW_E)$. As an example, consider the quasi-nilpotent operator $S \in B(l^2)$ defined by $S(x_1,x_2,x_3,...)=(0,\frac{x_1}{2},\frac{x_2}{3},....)$, then its adjoint $S^*(x_1,x_2,x_3,...)=(\frac{x_2}{3},\frac{x_3}{4},....)$. Since $\sigma_a(S)=\{0\}, \sigma_{uw}(S)=\{0\}$, $E(S)=\emptyset$, $S$ satisfies property $(UW_E)$. But $\sigma_a(S^*)=\{0\}, \sigma_{uw}(S^*)=\{0\}$, $E(S)=\{0\}$. So $S^*$ does not satisfies property $(UW_E)$. 
\section{Property $(UW_E)$ for direct sum and tensor product}
Let $T\in B(\mathcal{H})$ and $S\in B(\mathcal{K})$, then the direct sum of $T$ and $S$ is   given as $(T\oplus S)(x \oplus y)=Tx \oplus Sy \in \mathcal{H}\oplus\mathcal{K}$. The stability of variants of Weyl-type theorems under tensor sum has been studied by many authors \cite{berka,rashid,sanabri}. Now we study stability of  property $(UW_E)$ for direct sums. In general, Property $(UW_E)$ is not transferred to the direct sum if one of the operator satisfies property $(UW_E)$, as shown by the following examples.
\begin{example}
	$Q \in B(l^2)$ be a nilpotent operator and $T\in B(l^2)$ is given by
	$T(x_1,x_2,x_3,...)=(0,\frac{x_1}{2},\frac{x_2}{3},....)$. Then,
	$\sigma_a(Q)=\{0\}, \sigma_{uw}(Q)=\{0\}, E(Q)=\{0\}, 
	\sigma_a(T)=\{0\}, \sigma_{uw}(T)=\{0\}$, $E(T)=\emptyset, 
	\sigma_a(T\oplus Q)=\{0\}, \sigma_{uw}(T \oplus Q)=\{0\}$ and $E(T\oplus Q)=\{0\}$.
	Therefore, $Q \oplus T$ does not satisfies property $(UW_E)$.
\end{example}
\begin{example}
		Let $V$ be the Volterra operator defined on $L^2[0,1]$ given by 
	\begin{center}
		$V(f)(t)= \int_{t}^{1} f(x) \,dx , f \in L^2[0,1]$.
	\end{center} 
	We can see that $\sigma_{a}(V)=\sigma_{uw}(V)=\{0\}$, $E(V)= \emptyset$ and consider the operator $T \in B(l^2)$ defined by $T(x_1,x_2,x_3,...)=(0,x_1,0,0,...)$. Then $\sigma_{a}(T)=\sigma_{uw}(T)=\{0\}= E(T)$. So, $V$ satisfies property $(UW_E)$ whereas $T$ does not satisfies property $(UW_E)$. Now $\sigma_{a}(T \oplus V)=\sigma_{uw}(T \oplus V)=\{0\}=E(T\oplus V)$, $E_0(T\oplus V)=\emptyset$. Here $T\oplus V$ not satisfies property $(UW_E)$. 
\end{example}
The example below shows that even if operators $T$ and $S$ satisfy property $(UW_E)$, $T \oplus S$ does not satisfy that the property $(UW_E)$.
\begin{example}
Let $T,S \in B(l^2)$ is defined as $T(x_1,x_2,x_3,...)=(0,x_1,x_2,x_3,...)$ and $S(x_1,x_2,...)=(x_2,x_3,x_4,...)$. We have $\sigma_{a}(T)=\lbrace \lambda \in \mathbb{C}: |\lambda |=1\rbrace=\sigma_{uw}(T)$ and $E(T)= \emptyset$. Then $T$ satisfies property $(UW_E)$. Since $\sigma(S)=\sigma_{a}(S)=\sigma_{uw}(S)=\lbrace \lambda \in \mathbb{C}: |\lambda | \leq 1\rbrace$, $E(S)=\emptyset$. Hence $S$  satisfies property $(UW_E)$. Now  $\sigma_{a}(T \oplus S)= \{0\} \cup \lbrace \lambda \in \mathbb{C}: |\lambda | = 1\rbrace$, $\sigma_{uw}(T \oplus S)=\lbrace \lambda \in \mathbb{C}: |\lambda | = 1\rbrace$ and $E(T \oplus S)=\emptyset$. Thus $T \oplus S$ does not satisfies property $(UW_E)$.
%	Let $A$ be an injective quasi-nilpotent operator and $T=A \oplus A$. Then $\sigma_{a}(T)=\sigma_{uw}(T)=\{0\}$ and $E(A \oplus A)=\emptyset$. Also, $S=-A \oplus -A$. It can be seen that $T,S$ satisfies property $(UW_E)$. Now $T\oplus S$ does not satisfies property $(UW_E)$.
	\end{example}

%\begin{example}
%Let	$S\in B(l^2)$ defined by $S(a_1,a_2,a_3,....)=( \frac{a_2}{2},\frac{a_3}{3},\frac{a_4}{4},.....)$ and the adjoint is given by, $S^*(a_1,a_2,a_3,....)=(0, \frac{a_1}{2},\frac{a_2}{3},\frac{a_3}{4},.....)$. Then, 
%	$\sigma_a(S)=\sigma_{uw}(S)=E(S)=\{0\}$, $\sigma_a(S \oplus S)=\sigma_{uw}(S \oplus S)=\{0\}=E(S \oplus S)$. This shows that $S \oplus S$ does not satisfies property $(UW_E)$.
%	
%\end{example}
Let $\sigma_1(T)= \sigma_a(T) \setminus \sigma_{uw}(T)$ and $\sigma_2(T)= \sigma(T) \setminus \sigma_a(T)$. The following theorem shows that under certain conditions $T \oplus S$ satisfies property $(UW_E)$ for operators $T$ and $S$.

\begin{theorem}\label{T1}
	Suppose $T  \in B(\mathcal{H})$ has no isolated point in the spectrum and $S \in B(K)$ satisfies property $(UW_E)$. If $\sigma_{uw}(T\oplus S)= \sigma_a(T) \cup \sigma_{uw}(S)$ and $\sigma_2(T)=\emptyset$, then property $(UW_E)$ holds for $T\oplus S$.
\end{theorem}
\begin{proof}
	As 	$\sigma_a(T\oplus S)=\sigma_a(T) \cup \sigma_a(S)$	for any pair of operators $T $ and $S$, we have 
	\begin{equation*}
		\begin{split}
			\sigma_a(T\oplus S)\setminus \sigma_{uw}(T\oplus S) & =[\sigma_a(T) \cup \sigma_a(S)] \setminus[\sigma_a(T) \cup \sigma_{uw}(S)]\\
			& = \sigma_a(S) \setminus [\sigma_a(T)\setminus\sigma_{uw}(S)]\\
			&= [\sigma_a(S) \setminus \sigma_{uw}(S)]\setminus\sigma_a(T)\\
			& = E(S) \cap \rho_a(T).
		\end{split}
	\end{equation*}
	
	Since	$\text{iso}~ \sigma(T\oplus S)= \rm{iso}~\sigma(S) \cap \rho(T)$,
	\begin{equation*}
		\begin{split}
			E(T \oplus S)&= \text{iso}~  \sigma(T \oplus S) \cap \sigma_p(T \oplus S)\\
			& = [\text{iso}~ \sigma(S) \cap \rho(T)] \cap [\sigma_p(T)\cup \sigma_p(S)]\\
			& = [\text{iso}~ \sigma(S) \cap \rho(T) \cap \sigma_p(T)] \cup [\text{iso}~ \sigma(S) \cap \rho(T) \cap \sigma_p(S)]\\
			& = E(S)\cap \rho(T)\\
			&= E(S)\cap \rho_a(T).
		\end{split}
	\end{equation*}
	Thus, 	$\sigma_a(T \oplus S)\setminus\sigma_{uw}(T \oplus S)= E(T \oplus S)$. Hence $T \oplus S $ satisfies property $(UW_E)$.
	\end{proof}
%\begin{theorem}
%	Let $T \in B(\mathcal{H})$ and $S \in B(\mathcal{K})$ be such that $\sigma_p(S)=\sigma_p(T)$. If both $T$ and $S$ satisfies property $(UW_E)$ and $\sigma_{uw}(T \oplus S)= \sigma_{uw}(T)\cup\sigma_{uw}(S)$. Then $T \oplus S$ satisfies property $(UW_E)$.
%\end{theorem}
%\begin{proof}
%Since $\sigma_a(T \oplus S)=\sigma_a(T) \cup \sigma_a(S)$ for any operators $T$ and $S$,\\ 
%\begin{equation*}
%	\begin{split}
%	\sigma_a(T \oplus S)\setminus\sigma_{uw}(T \oplus S)	 & = [\sigma_a(T) \cup \sigma_a(S)] \setminus [\sigma_{uw}(T) \cup \sigma_{uw}(S)] \\
%		& = [E(T) \cap \rho_a(S)] \cup [E(S) \cap \rho_a(T)] \cup [E(T) \cap E(S)]\\
%		& = E(T) \cap E(S).
%	\end{split}
%\end{equation*}
%Also 
%\begin{equation*}
%E(T\oplus S)= \text{iso} \rm~ \sigma(T \oplus S) \cap \sigma_p(T \oplus S)= E(T) \cap E(S)
%\end{equation*}
%
%	Thus 	$\sigma_a(T \oplus S)\setminus\sigma_{uw}(T \oplus S)= E(T \oplus S)$. Hence $T \oplus S$ satisfies property $(UW_E)$.
%\end{proof}

\begin{corollary}
	Suppose $T \in B(\mathcal{H})$ is such that $\text{iso}~ \sigma
	(T)= \emptyset$ and  $S \in B(\mathcal{K})$ satisfies property $(UW_E)$ with $E(S)=\sigma_1(T \oplus S)=\sigma_2(T)= \emptyset$. Then $T\oplus S$ satisfies property $(UW_E)$.   
\end{corollary}
\begin{proof}
	Since $S\in(UW_E)$ and $E(S)=\emptyset$, $\sigma_a(S)=\sigma_{uw}(S)$. Since $\sigma_1(T \oplus S)=\emptyset$, it follows that $\sigma_a(T \oplus S)=\sigma_{uw}(T \oplus S)$. Also $\sigma_{uw}(T \oplus S)=\sigma_a(T)\cup\sigma_{uw}(S)=\sigma(T)\cup\sigma_{uw}(S)$. Then from Theorem \ref{T1}, $T\oplus S$ satisfies property $(UW_E)$.
	\end{proof}
\begin{corollary}
Let $T \in B(\mathcal{H})$ and $S \in B(\mathcal{K})$ such that $\text{iso}~ \sigma
(T)\cup\sigma_1(T)= \emptyset$ and  $S$ satisfies property $(UW_E)$. If $\sigma_{uw}(T\oplus S)=\sigma_{uw}(T)\cup\sigma_{uw}(S)$, then $T\oplus S$ satisfies property $(UW_E)$.
\end{corollary}
Tensor product of operators $T \in B(\mathcal{H} )$ and $S \in B(\mathcal{K})$ 
is defined by
\begin{eqnarray*}
	(T \otimes S)\ \Sigma_{i}(x_i \otimes y_i)&=& \Sigma_{i}Tx_i  \otimes Sy_i
\end{eqnarray*}  for each $\Sigma_{i}x_i \otimes y_i \in \mathcal{H} \otimes \mathcal{K}$. Stability of Weyl type theorems and its variations under tensor product has been studied by many authors, for more details see \cite{aponte,dug}.
 The following example shows that if one of the operators satisfies property $(UW_E)$, then $T \otimes S$ may not satisfies property $(UW_E)$.
\begin{example}
	Let $Q \in B(l^2)$ defined by $$Q(x_1,x_2,x_3,...)=(0,\frac{x_1}{2},\frac{x_2}{3},0,0,....),$$ a nilpotent operator  and $T$ be an injective quasi-nilpotent operator. We have $\sigma_{a}(Q)=\sigma_{uw}(Q)=E(Q)=\{0\}$ and $\sigma_{a}(T)=\sigma_{uw}(T)=\{0\}$, $E(T)=\emptyset$. Thus $T\in(UW_E)$, $Q\notin(UW_E)$. Also $\sigma_{a}(T \otimes Q)=\sigma_{a}(T)\sigma_{a}(Q)=\{0\}$. We know that $0\notin \sigma_{a}(T \otimes Q) \setminus \sigma_{uw}(T \otimes Q)$, $0 \in E(T \otimes Q)$. It follows that  $T \otimes Q$ does not satisfies property $(UW_E)$.
\end{example}
The following example shows that even if $T$ and $S$ satisfies property $(UW_E)$, tensor product $T \otimes S$ may not satisfies property $(UW_E)$.
\begin{example}
	Let $Q \in B(l^2)$ be an injective quasi-nilpotent operator and let
	\begin{center} $T=S=(I+Q) \oplus \alpha \oplus \beta$ on $B(l^2) \oplus \mathbb{C} \oplus \mathbb{C},$
			\end{center} where $\alpha \beta=1 \neq \alpha$. This example has been studied in \cite[Example 2]{dug} to show that property $(w)$ does not transfer from $T$ and $S$ to $T \otimes S$. We can see that $ \sigma_a(T)= \sigma_a(S)= \{1,\alpha,\beta\}, \sigma_{uw}(T)= \sigma_{uw}(S)= \{1\}$, $E(T)=\{\alpha,\beta\}=E(S)$, $\sigma_a(T \otimes S)= \{1,\alpha,\beta,\beta^{2},\alpha^{2}\}$, $\sigma_{uw}(T \otimes S)=\{1,\alpha,\beta\}$. Thus $T$ and $S$ satisfies property $(UW_E)$,
	$1 \notin \sigma_{a}(T \otimes S) \setminus \sigma_{uw}(T \otimes S)$ but
	$1 \in E(T \otimes S) $. So, $T \otimes S$ does not satisfies property $(UW_E)$.		
\end{example}
The following theorem shows that for isoloid operators $T$ and $S$ under certain conditions $T \otimes S$ satisfies property $(UW_E)$. This was studied for property $(V_E)$ in \cite{aponte} and for property $(w)$ in \cite{dug}.
\begin{theorem}
	\begin{enumerate}
\item[(i)] 	If $T$ and $S$ satisfies property $(UW_E)$, then $E(T \otimes S) \subseteq \sigma_a(T \otimes S) \setminus \sigma_{uw}(T \otimes S)$.
\item[(ii)] If $T$ and $S$ are isoloid operators and $0 \notin \sigma_p(T \otimes S)$, then $T \otimes S$ satisfies property $(UW_E)$ whenever $\sigma_{uw}(T \otimes S)=\sigma_a(T)\sigma_{uw}(S) \cup \sigma_a(S)\sigma_{uw}(T)$.
\end{enumerate}
\end{theorem}
\begin{proof}
	\begin{enumerate}
	\item[(i)] If $\mu \in E(T \otimes S) \subseteq E(T).E(S)$, then there exists $\lambda$ and $\nu$ in $E(T)$ and $E(S)$ such that $\mu= \lambda \nu$. Then, $\lambda \in \sigma_a(T) \setminus \sigma_{uw}(T)$ and $\nu \in \sigma_a(S) \sigma_{uw}(S)$. Hence $\mu \in \sigma_a(T \otimes S)$. Since $\sigma_{uw}(T \otimes S)=\sigma_a(T)\sigma_{uw}(S) \cup \sigma_a(S)\sigma_{uw}(T)$, $\mu \notin \sigma_{uw}(T \otimes S)$. Therefore,  $E(T \otimes S) \subseteq \sigma_a(T \otimes S) \setminus \sigma_{uw}(T \otimes S)$.\\
	\item[(ii)] Let $\lambda \in \sigma_a(T \otimes S) \setminus \sigma_{uw}(T \otimes S)$. Then $\lambda \neq 0$ and there exists $\mu \in \sigma_a(T) \setminus \sigma_{uw}(T)$ and $\nu \in \sigma_a(S) \setminus \sigma_{uw}(S)$ such that  $\lambda= \nu \mu$. This implies that $ \lambda \in \pi_{00}(T \otimes S) \subseteq E(T\otimes S)$ .
	Thus $ \sigma_a(T \otimes S)= \setminus \sigma_{uw}(T \otimes S) \subseteq E(T \otimes S)$. Hence $T \otimes S$ satisfies property $(UW_E)$.
\end{enumerate}
\end{proof}
\begin{remark}
	
From Theorem 3.9 and \cite[Thorem 1]{dug}, we see that if 
$T$ and $S$ are isoloid operators satisfying property $(UW_E)$ and $0 \notin \sigma_p(T \otimes S)$. Then $ T \otimes S\in (UW_E)$ is equivalent to $T \otimes S\in (w)$ and  $T \otimes S \in (a-Bt)$ whenever $\sigma_{uw}(T \otimes S)=\sigma_{a}(T )\sigma_{uw}(S )\cup\sigma_{uw}(T)\sigma_{a}(S )$.
\end{remark}
 The commutator of two operators $P$ and $Q$ is defined by $[P,Q]=PQ-QP$. If $Q_1$ and $Q_2$ are two quasi-nilpotent operators such that $[Q_1,A]=[Q_2,A]=0$ for some operators $A \in B(X) $ and $B \in B(Y)$ then, 
\begin{center}
	$(A+Q_1) \otimes (B+Q_2) = A \otimes B +Q$, 
\end{center}

where $Q= Q_1 \otimes B+ A\otimes Q_2+ Q_1 \otimes Q_2\in B(X \otimes Y)$ is a quasi-nilpotent operator. Duggal \cite{dug} studied the stabity of property $(w)$ under Perturbations by Riesz
operators. In \cite{aponte} Aponte, Sanabria, and V´asquez studied it for property $(V_E)$. In a similar manner, we observe that  
	let $Q_1 \in B(\mathcal{H})$ and $Q_2 \in B(\mathcal{K})$ such that $[Q_1,A]=[Q_2,B]=0$ for some operators $A \in B(\mathcal{H})$ and $B \in B(\mathcal{K})$. If $A\otimes B$ is isoloid and satisfies property $(UW_E)$, then $(A+Q_1) \otimes (B+Q_2)$  satisfies property $(UW_E)$. Also if $T,R\in B(\mathcal{H})$ be such that $T$ is isoloid and $R$ is a Riesz operator with $[T,R]=0$, $\sigma_{a}(T+R)=\sigma_{a}(T)$ and $\sigma(T+R)=\sigma(R)$. Then $T$ satisfies property $(UW_E)$ implies $T+R$ satisfies property $(UW_E)$. In addition, if $A$ and $B$ are two isoloid operators satisfying property $(UW_E)$ and $R_1$ and $R_2$ are Riesz operators such that $[A,R_1]=0=[B,R_2]$, $\sigma_a(A+R_1)=\sigma_{a}(A)$ , and $\sigma_{a}(B+R_2)=\sigma_{a}(B)$ and $\sigma(A+R_1)=\sigma(A)$ , and $\sigma(B+R_2)=\sigma(B)$ also $0 \notin \sigma_{p}(A \otimes B)$. Then, $A,B \in (UW_E) \Longrightarrow (A+R_1)\otimes (B+R_2) \in (UW_E)$ if and only if $A,B \in (w) \Longrightarrow (A+R_1)\otimes (B+R_2) \in (w)$.\\
	
\textbf{Acknowledgement} The research of first author is supported by senior research fellowship of university grants commission, India.

\end{document}